\documentclass[onecolumn,10pt]{article}
\usepackage[top=.75in, bottom=.75in, left=.75in, right=.75in]{geometry}
\usepackage{amsmath,amsfonts,amscd,amssymb,amsthm,bbm,bm}
\usepackage{graphicx}
\usepackage{epstopdf}
\usepackage{overpic}
\usepackage{cancel}
\usepackage{rotating}
\usepackage{url}
\usepackage{caption}
\usepackage{color}
\usepackage{rotating}
\usepackage{multirow}
\usepackage{wrapfig}
\usepackage{mathtools}
\usepackage{subeqnarray}
\usepackage{setspace}
\usepackage{palatino} 
\usepackage{microtype}
\usepackage[numbers,sort&compress]{natbib}

\usepackage[bottom,flushmargin,hang,multiple]{footmisc}
\usepackage{lipsum}
\newcommand\blfootnote[1]{%
  \begingroup
  \renewcommand\thefootnote{}\footnote{#1}%
  \addtocounter{footnote}{-1}%
  \endgroup
}

\definecolor{header1}{cmyk}{0,0,0,1}

\def \k {\mathbbm{k}}

\def \dim {\operatorname{dim}}

\def \R {\mathbbm{R}}
\def \C {\mathbbm{C}}

\numberwithin{equation}{section}
\numberwithin{table}{section}
\numberwithin{equation}{section}
\newtheorem{theorem}{Theorem}[section]

\newtheorem{proposition}[theorem]{Proposition}
\newtheorem{corollary}[theorem]{Corollary}
\newtheorem{definition}[theorem]{Definition}
\newtheorem{example}[theorem]{Example}

\newtheorem{remark}[theorem]{Remark}

\usepackage[utf8]{inputenc}

\usepackage[normalem]{ulem}
\usepackage{color}

\title{\vspace{-.125in}{\huge\selectfont \textbf{Eigenvalues and equivalence classes of \\ third-order symmetric tensors}}\vspace{-.075in}}

\author{\normalsize{Lishan Fang$^{1}$, Hua-Lin Huang$^{1*}$, Shengyuan Ruan$^{1}$, and Yu Ye$^{2}$}\\
\footnotesize{$^1$ School of Mathematical Sciences, Huaqiao University, Quanzhou 362021, China} \\
\footnotesize{$^{2}$ School of Mathematical Sciences, University of Science and Technology of China, Hefei 230026, China}
}

\date{}

\begin{document}
\maketitle



\blfootnote{$^*$ Corresponding author: hualin.huang@hqu.edu.cn}
\vspace{-.2in}
\begin{abstract}
This paper demonstrates that third-order real symmetric tensors cannot be classified up to equivalence by their eigenvalues only, thereby resolving a problem posed by Qi in 2006. By applying Harrison's center theory, we derive equivalence classes of $2 \times 2 \times 2$ symmetric tensors via the one-to-one correspondence with the canonical forms of their associated binary cubics. For such tensors, we compute the explicit characteristic polynomials and discover two previously unknown coefficients using the combination resultant. Pairs of third-order real symmetric tensors of all dimensions with identical eigenvalues but belonging to different equivalence classes are constructed to illustrate the inapplicability of eigenvalues for classification. 
\end{abstract}

\section{Introduction}\label{sec:intro}

The classification of symmetric tensors provides powerful tools for tackling higher-dimensional problems, with various applications in algebraic geometry, tensor networks, and quantum information; see~\cite{acuaviva2023minimal, dur2000three, landsberg2017geometry} and references therein. For symmetric matrices, eigenvalues play a key role in all kinds of classification problems~\cite{horn2012matrix}. Motivated by this, in~\cite{qi2006rank} Qi raised the following question: ``Can we classify third order and fourth order three-dimensional supersymmetric tensors by their eigenvalues?''. To date, this problem remains unresolved. In this work, we present examples to demonstrate that third-order real symmetric tensors cannot be classified by their eigenvalues.

The theory of eigenvalues is well established for matrices, and Qi~\cite{qi2005eigenvalues, qi2006rank} extended these concepts from the linear to the multilinear setting. Using a variational approach, Lim~\cite{lim2005singular} also independently established definitions for the eigenvalues of tensors. Since their introduction, tensor eigenvalues have become widely used in many fields, including spectral hypergraph theory, quantum physics, and tensor imaging~\cite{li2013characteristic, qi2018tensor}. Several definitions of tensor eigenvalues and characteristic polynomials were introduced in~\cite{qi2005eigenvalues} and further studied in~\cite{cartwright2013number, hu2013characteristic, li2013characteristic, ni2007degree, qi2006rank}. This work adopts the eigenvalues defined in~\cite{qi2006rank}, also known as the Z-eigenvalues in~\cite{qi2005eigenvalues}.

In this article, we apply Harrison's center theory~\cite{harrison1975grothendieck} introduced for homogeneous polynomials to tackle the associated classification problem. The center theory has been successfully applied to polynomial decoupling~\cite{huang2021diagonalizable, huang2022centres} and simultaneous decoupling~\cite{fang2025simultaneous}. 
It was also extended to study products of sums of powers~\cite{huang2022harrison} and to solve higher-degree algebraic equations~\cite{huang2025solving}. 
Building on this framework, we derive canonical forms for binary cubics by analyzing their centers. This yields the classification of their associated third-order two-dimensional symmetric tensors into equivalence classes. 

Using the method of combination resultants \cite{yang1996nonlinear, yuan2010computing, zhou1999preliminary}, we compute the explicit characteristic polynomial for a general $2 \times 2 \times 2$ real symmetric tensor, thereby confirming the form of the characteristic polynomial in \cite{cartwright2013number} and, in addition, determining two previously unknown coefficients. Examples of third-order real symmetric tensors in all dimensions are provided to show that symmetric tensors with identical eigenvalues can belong to different equivalence classes. This gives an answer to the problem of Qi for third-order symmetric tensors. In other words, one cannot classify symmetric tensors by their eigenvalues only, more information should be considered as well.

The remainder of this article is organized as follows. Section~\ref{sec:tensor_preliminary} describes symmetric tensors, eigenvalues, and center theory. Section~\ref{sec:equiv} determines equivalence classes of $2 \times 2 \times 2$ real symmetric tensors based on centers. Section~\ref{sec:tensor_charpoly} derives characteristic polynomials via the combination resultant. Section~\ref{sec:main} presents examples to illustrate that the classification solely by eigenvalues is infeasible. 

\section{Preliminaries}\label{sec:tensor_preliminary}

We begin by recalling the notions of tensors and eigenvalues in Subsections~\ref{sec:tensor} and~\ref{sec:eigenvalue}, respectively; more details are provided in~\cite{cartwright2013number, qi2005eigenvalues, qi2007eigenvalues}. Subsection~\ref{sec:center} reviews Harrison's center theory following~\cite{huang2021diagonalizable}.

\subsection{Real symmetric tensors}\label{sec:tensor}
 
Let $A=(a_{i_{1} \ldots i_{m}})\in T_{m,n}$ be a real $m$-th order $n$-dimensional tensor, i.e., a multi-array with entries $a_{i_{1} \ldots i_{m}} \in \R$, where $i_{j} = 1, \ldots, n$ for $j=1,\ldots,m$. If its entries are invariant under any permutation of indices, then $A$ is a symmetric tensor, which is also called a ``supersymmetric'' tensor in~\cite{qi2005eigenvalues, qi2006rank}. For clarity, we refer to an $m$-th order $n$-dimensional symmetric tensor as~$\underbrace{n \times n \times \cdots \times n}_{m \text{ times}}$ symmetric tensor.

A real homogeneous polynomial $f(x)\in\R[x_1,\ldots,x_n]$ of degree $m$ associated to $A$ is represented as
\[
    f(x) \equiv Ax^m :=\sum_{i_{1}, \ldots, i_{m}=1}^{n} a_{i_{1} \ldots i_{m}}x_{i_1}\ldots x_{i_m},
\]
where the vector $x=(x_1,\ldots,x_n)$ and $x^m$ is an $m$-th order $n$-dimensional rank-one tensor.
Two polynomials $f$ and $g$ are called equivalent, written $f \sim g$, if there exists an invertible matrix $P \in \operatorname{GL}_n(\mathbbm{R})$ such that 
\begin{equation*}
    g(x)=f(Px) = \sum_{j_1,\ldots, j_m =1}^{n} \sum_{i_1,\ldots, i_m =1}^{n} a_{i_{1} \ldots i_{m}} p_{i_1j_1}\ldots p_{i_mj_m} x_{j_1}\ldots x_{j_m}.
\end{equation*}
Obviously, equivalent polynomials can be transformed into the same canonical form. Similarly, if there exists~$P$ such that
\begin{equation*}
    B=AP^m = \left( \sum_{i_1,\ldots, i_m =1}^{n} a_{i_{1} \ldots i_{m}} p_{i_1j_1}\ldots p_{i_mj_m} \right)_{1 \le j_1,\ldots, j_m \le n},
\end{equation*}
$A$ and $B$ are called equivalent. It follows naturally that the equivalence classes of symmetric tensors correspond bijectively to the canonical forms of their associated homogeneous polynomials.

 

\subsection{Tensor Eigenvalues}\label{sec:eigenvalue}

Now we recall the eigenvalues and characteristic polynomials of tensors.

\begin{definition}\label{def:eigenvalue}
Let $A=(a_{i_{1} \ldots i_{m}}) \in T_{m,n}$. A number $\lambda \in \R$ and a vector $x\in \R^{n}$ are an eigenvalue and an eigenvector of~$A$ associated with~$\lambda$, respectively, if they are solutions of the following system:
 \begin{equation}\label{eqn:eigenvalue}
  \begin{cases} 
         Ax^{m-1} = \lambda x,  \\
        x^{T}x = 1.
  \end{cases}
 \end{equation}
\end{definition}

The resultant of System~\eqref{eqn:eigenvalue} is a polynomial in $\lambda$, which vanishes as long as System~\eqref{eqn:eigenvalue} has a nonzero solution. The polynomial is called the \textit{characteristic polynomial} of $A$ and denoted as~$\Psi_{A}(\lambda)$. Qi~\cite{qi2005eigenvalues} proved that the eigenvalues of $A$ are the roots of the characteristic polynomial~$\Psi_{A}(\lambda)$. Moreover, if $m$ is odd, $\Psi_{A}(\lambda)$ is a polynomial in $\lambda^2$ due to the symmetry of the eigenvalues under $\lambda \mapsto -\lambda$. 
  
Unlike Qi~\cite{qi2005eigenvalues} imposing the normalization condition $x^Tx=1$, Cartwright and Sturmfel.~\cite{cartwright2013number} consider eigenpairs instead. An \textit{eigenpair} consists of an eigenvalue and an associated eigenvector. Two eigenpairs $(\lambda, x)$ and $(\lambda', x')$ of the same tensor $A$ are said to be \textit{equivalent} if there exists a nonzero $t\in \C$ such that $t^{m-2}\lambda=\lambda'$ and $tx=x'$. This guarantees that the number of distinct equivalence classes of eigenpairs of a generic tensor is finite. In Qi~\cite{qi2005eigenvalues}, an upper bound of the number of eigenvalues is given as $n(m-1)^{n-1}-1$ when $m$ is even. Cartwright and Sturmfels~\cite{cartwright2013number} further improved it and showed that a generic $m$-th order $n$-dimensional tensor has~$((m-1)^{n}-1)/(m-2)$ equivalence classes of eigenpairs. This count also corresponds to the degree of $\Psi_{A}(\lambda)$.

\subsection{Harrison centers}\label{sec:center}

Harrison's center theory plays a crucial role in decoupling homogeneous polynomials into simpler forms, which can be utilized to derive canonical forms for their classification. Since there is a one-to-one correspondence between symmetric tensors and homogeneous polynomials, this method can be applied to classify symmetric tensors by analyzing their associated polynomials.

Let $f(x_1, \dots, x_n) \in \k[x_1, \dots, x_n]$ be a homogeneous polynomial over the ground field $\k$. Its center~$Z(f)$ is defined as
\begin{equation}\label{eqn:center}
      Z(f) =\left\{ X \in \k^{n \times n} \mid (H_fX)^T=H_fX \right\},
\end{equation}
where $H_f=\left(\frac{\partial^2 f}{\partial x_i \partial x_j}\right)_{1 \le i,\ j \le n}$ is the Hessian matrix of $f$.

The following theorem is useful for our further investigation.

\begin{theorem}\label{thm:center} 
Suppose $f(x_1, \dots, x_n) \in \k[x_1, \dots, x_n]$ is a homogeneous polynomial. Then 
\begin{itemize}
\item[(1)] The center $Z(f)$ is an associative subalgebra of $\k^{n\times n}$.
\item[(2)] If $f$ and $g$ are equivalent, then their centers $Z(f)$ and $Z(g)$ are isomorphic as algebras.
\item[(3)] If $f=f_1+\cdots+f_t$ is a sum of polynomials in disjoint sets of variables, then $Z(f) \cong Z(f_1) \times \cdots \times Z(f_t).$
\end{itemize}   
\end{theorem}

\begin{proof}
Detailed proof for all items can be found in~\cite{harrison1975grothendieck, huang2022centres}. Since items (2) and (3) will be applied later, we include a proof for them below for the convenience of the reader. 

(2) If $f$ is equivalent to $g$, then there exists $P \in \operatorname{GL}_n(\k)$ such that $g(x)=f(Px)$. Let $H_f$ and $H_g$ be the Hessian matrices of $f$ and $g$. It is clear that $H_f$ and $H_g$ are related by $P$, that is, $H_g=P^TH_fP$.
According to Equation~\eqref{eqn:center}, we have $(H_gY)^T=H_gY$ for any $Y\in Z(g)$. By substituting the relation $H_g=P^TH_fP$, the equation becomes \[(H_f PYP^{-1})^T = H_fPYP^{-1}.\] This implies that $PYP^{-1} \in Z(f)$ and consequently $Z(g)=P^{-1}Z(f)P=\{P^{-1}XP \mid \forall X \in Z(f)\}$. Therefore, if $f \sim g$, their centers are isomorphic as algebras, i.e., $Z(f) \cong Z(g)$.

(3) Suppose that $f$ is a direct sum $f=f_1+\cdots+f_t$, i.e. a sum of polynomials in disjoint sets of variables. Thus, the off-diagonal blocks of its Hessian matrix $H_f$ vanish and $H_f$ is a block-diagonal matrix expressed as $H_f=H_{f_1}\oplus\cdots\oplus H_{f_t}$. It is straightforward to see that $Z(f) \cong Z(f_1) \times \cdots \times Z(f_t)$ by Equation~\eqref{eqn:center}. 
\end{proof}

\section{Equivalence classes of $2 \times 2 \times 2$ symmetric tensors}\label{sec:equiv}

We classify the equivalence classes of $2 \times 2 \times 2$ real symmetric tensors by deriving the canonical forms for their associated two-dimensional real homogeneous cubic polynomials. The classification is achieved via Harrison's center theory.

Let $A=(a_{ijk})_{1\le i,j,k \le 2}$ be a $2 \times 2 \times 2$ real symmetric tensor expressed as
\begin{equation*}
A = \left( \begin{pmatrix}
a & b \\
b & c \\
\end{pmatrix} , 
\begin{pmatrix}
b & c \\
c & d \\
\end{pmatrix} \right),  
\end{equation*}
where $a_{111}=a$, $a_{112}=a_{121}=a_{211}=b$, $a_{122}=a_{212}=a_{221}=c$, $a_{222}=d$ and $a,b,c,d \in \R$. The tensor~$A$ corresponds to the homogeneous cubic polynomial~$f(x,y) \in \R[x_1,x_2]$, where
\begin{equation*}
    f(x,y) = ax^3+3bx^2y+3cxy^2+dy^3.
\end{equation*}

The center $Z(f)$ is computed by solving Equation~\eqref{eqn:center}, which reduces to
\begin{equation}\label{eqn:center222}
  \left\{
    \begin{alignedat}{4}
      & a x_{12} && + b x_{22} && = b x_{11} && + c x_{21}, \\
      & b x_{12} && + c x_{22} && = c x_{11} && + d x_{21},
    \end{alignedat}
  \right.
\end{equation}
where $X=(x_{ij})_{1\le i,j \le 2}$. It follows that $\dim Z(f)$ may be $4$, $3$ or $2$. Huang et al.~\cite{huang2022centres} showed that $f$ is degenerate if and only if $\dim Z(f)>2$. If $\dim Z(f)=2$, then $f$ is nondegenerate and $Z(f)$ forms a commutative subalgebra of $\R^{2\times 2}$. In this case, the center $Z(f)$ is isomorphic to $\R \times \R$, $\C$ or $\R[\epsilon]/(\epsilon^2)$. The canonical forms of $f$ are now obtained using their center algebra, as summarized below.

\begin{theorem}~\label{thm:standard222}
Any two-dimensional homogeneous cubic polynomial $f(x,y) \in \R[x_1,x_2]$ can be transformed into one of the following canonical forms under a change of variables.

\begin{itemize}
\item[(1)] $f=0$ if and only if $\dim Z(f)=4$.
\item[(2)] $f \sim x^3$ if and only if $\dim Z(f)=3$.
\item[(3)] $f \sim x^3+y^3$ if and only if $Z(f)\cong \R \times \R$.
\item[(4)] $f \sim x^3-3xy^2$ if and only if $Z(f)\cong \C$.
\item[(5)] $f \sim 3x^2y$ if and only if $Z(f)\cong \R[\epsilon]/(\epsilon^2)$.
\end{itemize}
\end{theorem}
\begin{proof}
Item (1) is obvious. The ``$\Rightarrow$" parts of items (2-5) follow from the direct computation of Equation~\eqref{eqn:center}. Now we include a proof of the ``$\Leftarrow$" parts for items (2-5) below.

(2) If $\dim Z(f) = 3$, two equations in System~\eqref{eqn:center222} are linearly dependent. This implies that $(b,c,d) = \alpha (a,b,c)$, or $\alpha (b,c,d) = (a,b,c)$ for some constant~$\alpha$. Assuming the former without loss of generality, then we have
\begin{equation*}
    f(x,y) = ax^3+3\alpha ax^2y+3\alpha^2 axy^2+\alpha^3 ay^3 = a(x+\alpha y)^3.
\end{equation*}
It is obvious that $f \sim x^3$.

(3) If $Z(f)\cong \R \times \R$, then under a suitable change of variable we may assume
\begin{equation*}
    Z(f)=\left\{ \begin{pmatrix}
        s & 0 \\ 0 & t
    \end{pmatrix} \vline s,t \in \R \right\}.
\end{equation*}
We can see that $b=c=0$ from System~\eqref{eqn:center222} and $f(x,y) = ax^3+dy^3$, where $ad \ne 0$. Therefore, we have $f \sim x^3+y^3$.

(4) For $Z(f)\cong \C$, we may assume that
\begin{equation*}
    Z(f)=\left\{ \begin{pmatrix}
        s & -t \\ t & s
    \end{pmatrix} \vline s,t \in \R \right\},
\end{equation*}
it is easy to see $c=-a$ and $d=-b$ from System~\eqref{eqn:center222}. Choose a complex number $u+vi$ such that $(u+vi)^3=a+bi$. Let $p(x,y)=ux+vy$ and $q(x,y)=vx-uy$. We obtain
\begin{equation*}
\begin{split}
f(x,y) & =  ax^3+3bx^2y-3axy^2-by^3 \\
 & = \frac{1}{2} (a+bi)(x-yi)^3+\frac{1}{2} (a-bi)(x+yi)^3 \\
 & = \frac{1}{2} \left(p(x,y)+q(x,y)i\right)^3+\frac{1}{2} \left(p(x,y)-q(x,y)i\right)^3 \\
 & = p(x,y)\left(p(x,y)^2-3q(x,y)^2\right).
\end{split}
\end{equation*}
Now we can see that $f \sim x^3-3xy^2$.

(5) For $Z(f)\cong \R[\epsilon]/(\epsilon^2)$, we assume that
\begin{equation*}
    Z(f)=\left\{ \begin{pmatrix}
        s & 0 \\ t & s
    \end{pmatrix} \vline s,t \in \R \right\},
\end{equation*}
this gives $c=d=0$ according to System~\eqref{eqn:center222}. Thus, $f(x,y)=ax^3+3bx^2y=x^2(ax+3by)$ and $f \sim 3x^2y$.
\end{proof}

\begin{corollary}\label{cor:standard}
The equivalence classes of $2 \times 2 \times 2$ real symmetric tensors are listed below, corresponding to the canonical forms of their associated homogeneous polynomials in Theorem~\ref{thm:standard222}.

\begin{itemize}
\item[(1)] $\left( \begin{pmatrix}
0 & 0 \\
0 & 0 \\
\end{pmatrix} ,
\begin{pmatrix}
0 & 0 \\
0 & 0 \\
\end{pmatrix} \right)$.
\item[(2)] $\left( \begin{pmatrix}
1 & 0 \\
0 & 0 \\
\end{pmatrix} ,
\begin{pmatrix}
0 & 0 \\
0 & 0 \\
\end{pmatrix} \right)$.
\item[(3)] $\left( \begin{pmatrix}
1 & 0 \\
0 & 0 \\
\end{pmatrix} ,
\begin{pmatrix}
0 & 0 \\
0 & 1 \\
\end{pmatrix} \right)$. 
\item[(4)] $\left( \begin{pmatrix}
1 & 0 \\
0 & -1 \\
\end{pmatrix} ,
\begin{pmatrix}
0 & -1 \\
-1 & 0 \\
\end{pmatrix} \right)$. 
\item[(5)] $\left( \begin{pmatrix}
0 & 1 \\
1 & 0 \\
\end{pmatrix} , 
\begin{pmatrix}
1 & 0 \\
0 & 0 \\
\end{pmatrix} \right)$. 
\end{itemize}
\end{corollary}

\begin{remark}\label{rmk:complex}
Over the field $\C$ of complex numbers, the canonical forms of two-dimensional homogeneous cubic polynomials and the equivalence classes of $2 \times 2 \times 2$ symmetric tensors are identical to those in the real cases, except that item (4) is merged into item (3). The difference arises from the two-dimensional commutative subalgebraic structures of $\C^{2\times 2}$, which are isomorphic either to $\C \times \C $ or $\C[\epsilon]/(\epsilon^2)$. In contrast to the real case, there is a field extension of degree two over $\R$ which corresponds to item (4), while this is not the case since $\C$ is algebraically closed. 
\end{remark}

\section{Eigenvalues of $2 \times 2 \times 2$ symmetric tensors}\label{sec:tensor_charpoly}

Let $A$ be a general $2 \times 2 \times 2$ real symmetric tensor, where
\begin{equation*}
A = \left( \begin{pmatrix}
a & b \\
b & c \\
\end{pmatrix} , 
\begin{pmatrix}
b & c \\
c & d \\
\end{pmatrix} \right).  
\end{equation*}
According to Definition~\ref{def:eigenvalue}, eigenvalues $\lambda \in \R$ and associated eigenvectors of $A$ must satisfy
\begin{equation}\label{eqn:charac_nonlinear}
  \begin{cases} 
    \begin{aligned}
      ax^2+2bxy+cy^2-\lambda x &= 0, \\
      bx^2+2cxy+dy^2-\lambda y &= 0,
    \end{aligned} \\
    x^2+y^2-1 = 0.
  \end{cases}
\end{equation}
The characteristic polynomial~$\Psi_{A}(\lambda)$ is a polynomial associated with $A$, whose roots are the eigenvalues of $A$.

In the following, we calculate the characteristic polynomial~$\Psi_{A}(\lambda)$ by eliminating $x, y$ from System~\eqref{eqn:charac_nonlinear} using the combination resultant proposed in~\cite{zhou1999preliminary}, which was developed to eliminate variables from systems of three quadratic equations. Similar techniques have been applied to compute the resultant~\cite{cox2005using}, Bezout matrix~\cite{yang1996nonlinear}, and Dixon matrix~\cite{yuan2010computing} for nonlinear polynomial systems. For the convenience of the reader, we give a brief description of the method below.

Consider a system of three quadratic equations 
\begin{equation}\label{eqn:nonlinear}
  \left\{
    \begin{alignedat}{7}
      F_{1}(x,y) &= a_1x^2 &&+ a_2xy &&+ a_3y^2 &&+ a_4x &&+ a_5y &&+ a_6 &&= 0, \\
      F_{2}(x,y) &= b_1x^2 &&+ b_2xy &&+ b_3y^2 &&+ b_4x &&+ b_5y &&+ b_6 &&= 0, \\
      F_{3}(x,y) &= c_1x^2 &&+ c_2xy &&+ c_3y^2 &&+ c_4x &&+ c_5y &&+ c_6 &&= 0,
    \end{alignedat}
  \right.
\end{equation}
where the coefficients $a_{j}, b_{j}, c_{j}\in \C$ for $1 \le j \le 6$. Note that $F_{i}(x,y)$ can be treated as a homogeneous linear polynomial with the variables $(x^2,xy,y^2,x,y,1)$, and System~\eqref{eqn:nonlinear} becomes a system of linear equations. More precisely, given a solution $(x_0, y_0)$ of the previous system, then $(x_0^2,x_0y_0,y_0^2,x_0,y_0,1)$ is a solution to the following system of linear equations
\begin{equation*}
  \left\{
    \begin{alignedat}{7}
      & a_1z_1 &&+ a_2z_2 &&+ a_3z_3 &&+ a_4z_4 &&+ a_5z_5 &&+ a_6z_6 &&= 0, \\
      & b_1z_1 &&+ b_2z_2 &&+ b_3z_3 &&+ b_4z_4 &&+ b_5z_5 &&+ b_6z_6 &&= 0, \\
      & c_1z_1 &&+ c_2z_2 &&+ c_3z_3 &&+ c_4z_4 &&+ c_5z_5 &&+ c_6z_6 &&= 0.
    \end{alignedat}
  \right.
\end{equation*}
It is well known that a homogeneous linear system with an equal number of unknowns and equations has nonzero solutions if and only if the determinant of the corresponding coefficient matrix equals zero. 
Since the three equations in System~\eqref{eqn:nonlinear} are insufficient to form a square matrix for the determinant, we generate three additional quadratic equations that vanish at all common zeros of System~\eqref{eqn:nonlinear}. 

We begin by decomposing each polynomial $F_{i}(x,y)$ into three components, where $F_{i}(x,y) = \sum_{k=1}^{3}f_{ik}(x,y)$ for $i=1,2,3$. A $3\times 3$ determinant is constructed as
\begin{equation*}\label{eqn:nonlinear_det}
F(x,y) =  
\begin{vmatrix}
    f_{11}(x,y) & f_{12}(x,y) & f_{13}(x,y) \\ 
    f_{21}(x,y) & f_{22}(x,y) & f_{23}(x,y) \\ 
    f_{31}(x,y) & f_{32}(x,y) & f_{33}(x,y) 
\end{vmatrix}
=\begin{vmatrix}
    f_{11}(x,y) & f_{12}(x,y) & F_1(x,y) \\ 
    f_{21}(x,y) & f_{22}(x,y) & F_2(x,y) \\ 
    f_{31}(x,y) & f_{32}(x,y) & F_3(x,y) 
\end{vmatrix}.
\end{equation*}
It is clear that the zeros of $F(x,y)$ contain the common zeros of System~\eqref{eqn:nonlinear}. 
We select three specific combinations to construct the three additional quadratic equations. First, the determinant is defined as 
\begin{equation*}
\overline{F}_4(x,y) = \begin{vmatrix}
    a_1x^2 & a_2xy+a_3y^2 & a_4x+a_5y+a_6 \\ 
    b_1x^2 & b_2xy+b_3y^2 & b_4x+b_5y+b_6 \\ 
    c_1x^2 & c_2xy+c_3y^2 & c_4x+c_5y+c_6 
\end{vmatrix}.
\end{equation*}
By factoring common terms $x^2$ and $y$ from the first and second columns, respectively, we obtain a quadratic equation
\begin{equation*}
    F_4(x,y)  =  
    \begin{vmatrix}
        a_1 & a_2x+a_3y & a_4x+a_5y+a_6 \\ 
        b_1 & b_2x+b_3y & b_4x+b_5y+b_6 \\ 
        c_1 & c_2x+c_3y & c_4x+c_5y+c_6 
    \end{vmatrix} 
    = \tau_{124}x^2+(\tau_{125}+\tau_{134})xy+\tau_{135}y^2+\tau_{126}x+\tau_{136}y,
\end{equation*}
where 
$
\tau_{ijk} =  
\begin{vmatrix}
    a_i & a_j & a_k \\ 
    b_i & b_j & b_k \\ 
    c_i & c_j & c_k
\end{vmatrix},
$ 
$1 \le i,j,k \le 6$, $i \ne j$, $j \ne k$ and $k \ne i$. Next, we construct two more determinants with different combinations
\begin{equation*}
\overline{F}_5(x,y) =  
\begin{vmatrix}
    a_6 & a_3y^2+a_5y & a_1x^2+a_2xy+a_4x \\ 
    b_6 & b_3y^2+b_5y & b_1x^2+b_2xy+b_4x \\ 
    c_6 & c_3y^2+c_5y & c_1x^2+c_2xy+c_4x 
\end{vmatrix}, \quad
\overline{F}_6(x,y) =  
\begin{vmatrix}
    a_3y^2 & a_1x^2+a_2xy & a_4x+a_5y+a_6 \\ 
    b_3y^2 & b_1x^2+b_2xy & b_4x+b_5y+b_6 \\ 
    c_3y^2 & c_1x^2+c_2xy & c_4x+c_5y+c_6 
\end{vmatrix}.
\end{equation*}
After canceling common factors, we obtain another two quadratic equations as
\begin{alignat*}{1}
    F_5(x,y) &=  
    \begin{vmatrix}
        a_6 & a_3y+a_5 & a_1x+a_2y+a_4 \\ 
        b_6 & b_3y+b_5 & b_1x+b_2y+b_4 \\ 
        c_6 & c_3y+c_5 & c_1x+c_2y+c_4 
    \end{vmatrix} 
    = -\tau_{136}xy-\tau_{236}y^2-\tau_{156}x+(\tau_{346}-\tau_{256})y-\tau_{456}, \\
    F_6(x,y) &=  
    \begin{vmatrix}
        a_3 & a_1x+a_2y & a_4x+a_5y+a_6 \\ 
        b_3 & b_1x+b_2y & b_4x+b_5y+b_6 \\ 
        c_3 & c_1x+c_2y & c_4x+c_5y+c_6 
    \end{vmatrix} 
    = \tau_{314}x^2+(\tau_{315}+\tau_{324})xy+\tau_{325}y^2+\tau_{316}x+\tau_{326}y.
\end{alignat*}
As before, it is ready to see that the common zeros of $F_4(x,y),F_5(x,y)$ and $F_6(x,y)$ contain those of System~\eqref{eqn:nonlinear}. 

Then $\{ F_{i} (x, y) = 0 \mid 1 \le i \le 6\}$ becomes a system of homogeneous linear equations in variables $z_1, \dots, z_6$ with coefficient matrix
\begin{equation*}
C =  
\begin{pmatrix}
     a_1 & a_2 & a_3 & a_4 & a_5 & a_6 \\ 
     b_1 & b_2 & b_3 & b_4 & b_5 & b_6\\ 
     c_1 & c_2 & c_3 & c_4 & c_5 & c_6 \\ 
     \tau_{124} & \tau_{125}+\tau_{134} & \tau_{135} & \tau_{126} & \tau_{136} & 0 \\ 
     0 & -\tau_{136} & -\tau_{236} & -\tau_{156} & \tau_{346}-\tau_{256} & -\tau_{456} \\ 
     \tau_{314} & \tau_{315}+\tau_{324} & \tau_{325} & \tau_{316}& \tau_{326} & 0
\end{pmatrix}.
\end{equation*}
Thus, if System~\eqref{eqn:nonlinear} has nonzero solutions, then $\det(C)=0$. In accordance with \cite{zhou1999preliminary}, the determinant $\det(C)$ is called a combination resultant of System~\eqref{eqn:nonlinear}.

Now we apply the previous process to System \eqref{eqn:charac_nonlinear} and obtain its combination resultant
\begin{equation}\label{eqn:char_det}
\Psi_{A}(\lambda) = 
\begin{vmatrix}
     a & 2b & c & -\lambda & 0 & 0 \\ 
     b & 2c & d & 0 & -\lambda & 0 \\ 
     1 & 0 & 1 & 0 & 0 & -1 \\ 
     2c\lambda & (-3b+d)\lambda & (a-c)\lambda & 2(b^2-ac) & bc-ad & 0 \\ 
     0 & ad-bc & 2(bd-c^2) & -a\lambda & -(d+2b)\lambda & \lambda^2 \\ 
     (b-d)\lambda & (3c-a)\lambda & -2b\lambda & ad-bc & 2bd-2c^2 & 0 
\end{vmatrix}.
\end{equation}
This turns out to be the desired characteristic polynomial of the general $2 \times 2 \times 2$ real symmetric tensor.

\begin{theorem}~\label{thm:charac_poly222}
Let $A=(a_{ijk})_{1\le i,j,k \le 2}$ be a $2 \times 2 \times 2$ real symmetric tensor with 
\begin{equation*}
    a_{111}=a, \quad a_{112} = a_{121} = a_{211}=b, \quad a_{122} = a_{212} = a_{221}=c, \quad a_{222}=d.
\end{equation*}
The characteristic polynomial $\Psi_{A}(\lambda)$ of $A$ takes the form of
\begin{equation*}
    \Psi_{A}(\lambda) = \alpha_2\lambda^6 + \alpha_4\lambda^4 + \alpha_6\lambda^2 + \alpha_8,
\end{equation*}
where
\begin{align*}
    \alpha_2 & = -(-a+3c)^2-(3b-d)^2, \\
    \alpha_4 &= a^4+24b^4-6a^3c+24c^4-8b^3d+12c^2d^2+d^4+9b^2(5c^2+d^2) \\
     & +3a^2(4b^2+3c^2-2bd+d^2)-2ac(6b^2+4c^2+6bd+3d^2) -6b(2c^2d+d^3),  \\
     \alpha_6 &= -2a^4d^2+12a^3bcd+8a^3c^3+6a^3cd^2-8a^2b^3d-42a^2b^2c^2-12a^2b^2d^2 +12a^2bc^2d\\
     & +6a^2bd^3-24a^2c^4 -12a^2c^2d^2 -2a^2d^4+48ab^4c +30ab^2c^3 +12ab^2cd^2+12abcd^3-8ac^3d^2\\
     & -16b^6-12b^4c^2-24b^4d^2+30b^3c^2d+8b^3d^3-12b^2c^4-42b^2c^2d^2+48bc^4d-16c^6, \\
     \alpha_8 &= \left(a^2d^2-6abcd+4ac^3+4b^3d-3b^2c^2\right)^2.
\end{align*}
\end{theorem}

\begin{proof}
First of all, if $\lambda_0$ is an eigenvalue of the tensor $A,$ then $\lambda_0$ is a root of the combination resultant $\Psi_{A}(\lambda)$ by the preceding argument. The coefficients of $\Psi_{A}(\lambda)$ are obtained by calculating the determinant in Equation~\eqref{eqn:char_det} directly as above. Secondly, by multiplying the first equation of \eqref{eqn:charac_nonlinear} by $y$, the second by $x$, and subtracting them, we obtain the following equation \[ax^2y+2bxy^2+cy^3=bx^3+2cx^2y+dxy^2.\] It is clear that a binary cubic has three zeros in terms of homogeneous coordinates. In cooperation with the third equation of \eqref{eqn:charac_nonlinear}, there are six eigenvectors and associated eigenvalues for $A$ in general. Finally, by the definition Equation \eqref{def:eigenvalue} of tensor eigenvalues, it is easy to observe that eigenvalues $\lambda_0$ and $-\lambda_0$ appear in pairs for odd order tensors, thus the characteristic polynomial is a polynomial in $\lambda^2$. In summary, since $\Psi_{A}(\lambda)$ contains all eigenvalues of $A$ and has the exact degree, it follows that $\Psi_{A}(\lambda)$ is the desired characteristic polynomial of the general $2 \times 2 \times 2$ real symmetric tensor $A$.
\end{proof}

\begin{remark}\label{rmk:charpoly} 
Cartwright and Sturmfels~\cite{cartwright2013number} presented an explicit characteristic polynomial~$\psi(\lambda) = C_2\lambda^6 + C_4\lambda^4 + C_6\lambda^2 + C_8$ for a general $2 \times 2 \times 2$ tensor and computed $C_2,C_8$ while $C_4, C_6$ remain unknown. If the $2 \times 2 \times 2$ tensor is required to be symmetric, then their $C_2, C_8$ coincide with ours. In addition, we determine the two previously unknown coefficients $C_4, C_6$ in this case.
\end{remark}

An explicit example for computing a characteristic polynomial and corresponding eigenvalues for a $2 \times 2 \times 2$ real symmetric tensor is given below.

\begin{example}\label{ex:equiv1}
\emph{Let $A = \left( \begin{pmatrix}
1 & 1 \\
1 & 1 \\
\end{pmatrix} , 
\begin{pmatrix}
1 & 1 \\
1 & 1 \\
\end{pmatrix} \right)$. According to Definition~\ref{def:eigenvalue}, its eigenvalues~$\lambda \in \R$ satisfy 
 \begin{equation*}
  \begin{cases} 
    x^2+2xy+y^2=\lambda x, \\
    x^2+2xy+y^2=\lambda y, \\
    x^2+y^2=1.
  \end{cases}
 \end{equation*}
The characteristic polynomial of $A$ is calculated using Theorem~\ref{thm:charac_poly222} as
\begin{equation*}
    \Psi_{A}(\lambda) = -8\lambda^6+64\lambda^4 = -8(\lambda^2)^2(\lambda^2-8).
\end{equation*}
Since the order of $A$ is odd, the roots of the characteristic polynomial in terms of $\lambda^2$ are 0 and 8, which shows that the eigenvalues of $A$ are $\lambda_{1}^2 = \lambda_{2}^2 = 0$ and $\lambda_{3}^2 = 8$.}
\end{example}

\section{Counterexamples to classification by eigenvalues}\label{sec:main}

This section presents three pairs of real third-order symmetric tensors that have identical eigenvalues but belong to different equivalence classes. Thus, it is evident that the eigenvalues are insufficient to classify real symmetric tensors into distinct equivalence classes. The pairs of two-, three- and four-dimensional examples are given in Subsections~\ref{sec:examples_222},~\ref{sec:examples_333} and~\ref{sec:examples_444}, respectively. We then generalize this result to all dimensions.

\subsection{$2 \times 2 \times 2$ symmetric tensors}\label{sec:examples_222}

\begin{example}\label{ex:222_1}
\emph{Let
$A_1 = \left( \begin{pmatrix}
1 & 0 \\
0 & -1 \\
\end{pmatrix} , 
\begin{pmatrix}
0 & -1 \\
-1 & 0 \\
\end{pmatrix} \right)$
be a $2 \times 2 \times 2$ real symmetric tensor. According to Definition~\ref{def:eigenvalue}, its eigenvalues~$\lambda \in \R$ satisfy 
 \begin{equation}\label{eqn:ex1_system}
  \begin{cases} 
    x^2-y^2=\lambda x, \\
    -2xy=\lambda y, \\
    x^2+y^2=1.
  \end{cases}
 \end{equation}
By Theorem~\ref{thm:charac_poly222}, we calculate the characteristic polynomial as
\begin{equation*}
    \Psi_{A_1}(\lambda) = -16\lambda^6+48\lambda^4-48\lambda^2+16 = -16(\lambda^2-1)^3.
\end{equation*}
Thus, the eigenvalues of $A_1$ are $\lambda_{1}^2=\lambda_{2}^2=\lambda_{3}^2=1$.}
\end{example}

\begin{example}\label{ex:222_2}
\emph{Let 
$A_2=\left( \begin{pmatrix}
1 & 0 \\
0 & \frac{1}{2} \\
\end{pmatrix} , 
\begin{pmatrix}
0 & \frac{1}{2} \\
\frac{1}{2} & 0 \\
\end{pmatrix} \right)$
be a $2 \times 2 \times 2$ real symmetric tensor. According to Definition~\ref{def:eigenvalue}, its eigenvalues~$\lambda \in \R$ are the solutions of
 \begin{equation}\label{eqn:ex2_system}
  \begin{cases} 
    x^2+\frac{1}{2}y^2=\lambda x, \\
    xy=\lambda y, \\
    x^2+y^2=1.
  \end{cases}
 \end{equation}
The characteristic polynomial is then calculated using Theorem~\ref{thm:charac_poly222} as
\begin{equation*}
    \Psi_{A_2}(\lambda) = -\frac{1}{4}\lambda^6+\frac{3}{4}\lambda^4-\frac{3}{4}\lambda^2+\frac{1}{4}=-\frac{1}{4}(\lambda^2-1)^3.
\end{equation*}
Hence, the eigenvalues of $A_2$ are $\lambda_{1}^2=\lambda_{2}^2=\lambda_{3}^2=1$.}
\end{example}

Though $A_1$ and $A_2$ share the same eigenvalues, we have the following observation.

\begin{proposition}\label{prop:tensor222}
The tensors $A_1$ and $A_2$ are not equivalent.
\end{proposition}

\begin{proof}
The tensor $A_1$ is associated with the polynomial~$f_1(x,y)=x^3-3xy^2$, while the tensor $A_2$ is associated with the polynomial $f_2(x,y)=x^3+\frac{3}{2}xy^2$. It is clear that $f_1$ is the canonical form of type (4) in Theorem~\ref{thm:standard222}. As for $f_2$, its Hessian matrix is
\begin{equation*}
    H_{f_2}=\left(
    \begin{array}{cc}
    		6x & 3y \\
    		3y & 3x \\
    	\end{array}
    	\right),
\end{equation*}
and we compute the center by Equation~\eqref{eqn:center} as 
\begin{equation*}
    Z(f_2)=\left\{ \begin{pmatrix}
        s & t \\ 2t & s
    \end{pmatrix} \vline s,t \in \R \right\},
\end{equation*}
which implies that $Z(f_{2})\cong \R \times \R$. Thus, $f_2 \sim x^3+y^3$ by Theorem~\ref{thm:standard222}. Therefore, $f_1$ and $f_2$ are not equivalent, which suggests that their associated tensors~$A_1$ and $A_2$ belong to different equivalence classes in Corollary~\ref{cor:standard}. 
\end{proof}

\subsection{$3 \times 3 \times 3$ symmetric tensors}\label{sec:examples_333}


Now we continue to examine examples with higher dimensions. For $n>2$, we construct a $n \times n \times n$ symmetric tensor $A=(a_{ijk})_{1\le i,j,k\le n}$ as a direct sum of a $2 \times 2 \times 2$ symmetric tensor $B=(b_{ijk})_{1\le i,j,k\le 2}$ and a third-order $(n-2)$-dimensional unit tensor $I_{(n-2)}$, that is, $A=B\oplus I_{(n-2)}$. 
Note that an $m$-th order $n$-dimensional tensor is called an $m$-th \textit{unit tensor} if its entries satisfy
 \begin{equation*}
  I_{i_1 \ldots i_m} = 
  \begin{cases} 
        1, \quad \text{if } i_1 = \cdots = i_m, \\
        0, \quad \text{otherwise},
  \end{cases}
 \end{equation*}
for $ i_1, \ldots, i_m = 1, \ldots, n$. Thus, $a_{ijk}=b_{ijk}$ for $1 \le i,j,k \le 2$, $a_{iii}=1$ for $3 \le i \le n$ and $a_{ijk}=0$ elsewhere. In the rest of this section, we denote an $n\times n \times n$ unit tensor as $I_{(n)}$ for clarity.

\begin{example}\label{ex:333_1}
\emph{We construct the $3 \times 3 \times 3$ real symmetric tensor $A_3$ as $A_3=A_1 \oplus I_{(1)}$, where
\begin{equation*}
A_3=\left( \begin{pmatrix}
1 & 0 & 0\\
0 & -1 & 0 \\
0 & 0 & 0\\
\end{pmatrix} , \begin{pmatrix}
0 & -1 & 0\\
-1 & 0 & 0 \\
0 & 0 & 0\\
\end{pmatrix},
\begin{pmatrix}
0 & 0 & 0\\
0 & 0 & 0 \\
0 & 0 & 1\\
\end{pmatrix} \right).
\end{equation*}
According to Definition~\ref{def:eigenvalue}, we obtain eigenvalues $\lambda \in \R$ by solving
\begin{equation}\label{eqn:ex3_system}
  \begin{cases} 
    x^2-y^2=\lambda x, \\
    -2xy=\lambda y, \\
    z^2 = \lambda z, \\
    x^2+y^2+z^2=1.
  \end{cases}
 \end{equation}}

\emph{If $z = 0$, System~\eqref{eqn:ex3_system} is equivalent to System~\eqref{eqn:ex1_system} and so we obtain three eigenvalues $\lambda_{1}^2=\lambda_{2}^2=\lambda_{3}^2=1$.
If $z \ne 0$, it is easy to see that $z = \lambda$ and System~\eqref{eqn:ex3_system} reduces to a system with only two variables. Note that this new system is the same as System~\eqref{eqn:ex1_system} except for a different normalization equation $x^2+y^2+\lambda^2=1$. We calculate the combination resultant to get $16(\lambda^2-1)(2\lambda^2-1)^3$ which offers four other eigenvalues as $\lambda_4^2=1, \ \lambda_5^2=\lambda_6^2=\lambda_7^2=\frac{1}{2}.$ }
\end{example}

\begin{example}\label{ex:333_2}
\emph{We construct the $3 \times 3 \times 3$ real symmetric tensor $A_4$ as $A_4=A_2 \oplus I_{(1)}$, where
\begin{equation*}
A_4=\left( \begin{pmatrix}
1 & 0 & 0\\
0 & \frac{1}{2} & 0 \\
0 & 0 & 0\\
\end{pmatrix} , \begin{pmatrix}
0 & \frac{1}{2} & 0\\
\frac{1}{2} & 0 & 0 \\
0 & 0 & 0\\
\end{pmatrix},
\begin{pmatrix}
0 & 0 & 0\\
0 & 0 & 0 \\
0 & 0 & 1\\
\end{pmatrix} \right).
\end{equation*}
According to Definition~\ref{def:eigenvalue}, its eigenvalues $\lambda \in \R$ satisfy 
 \begin{equation}\label{eqn:ex4_system}
  \begin{cases} 
    x^2+\frac{1}{2}y^2=\lambda x, \\
    xy=\lambda y, \\
    z^2 = \lambda z, \\
    x^2+y^2+z^2=1.
  \end{cases}
 \end{equation}}

\emph{If $z = 0$, System~\eqref{eqn:ex4_system} is equivalent to System~\eqref{eqn:ex2_system} and this provides three eigenvalues $\lambda_{1}^2=\lambda_{2}^2=\lambda_{3}^2=1$.
If $z \ne 0$, then $z = \lambda$ and System~\eqref{eqn:ex4_system} only has two variables. By computing the corresponding combination resultant as $\frac{1}{4}(\lambda^2-1)(2\lambda^2-1)^3$ and we get four other eigenvalues as $\lambda_4^2=1, \ \lambda_5^2=\lambda_6^2=\lambda_7^2=\frac{1}{2}.$ } 
\end{example}

The above examples show that the eigenvalues of $A_3$ and $A_4$ are identical, however we have

\begin{proposition}\label{prop:tensor333}
The tensors $A_3$ and $A_4$ are not equivalent.
\end{proposition}


\begin{proof}
Note that $A_3$ is associated to $f_3(x,y,z)=x^3-3xy^2+z^3$ and $A_4$ is associated to $f_4(x,y,z)=x^3+\frac{3}{2}xy^2+z^3$. According to Theorem \ref{thm:center}, it is easy to get $Z(f_3) \cong \C \times \R$ and $Z(f_4) \cong \R \times \R \times \R$ by direct calculation. It follows that the symmetric tensors $A_3$ and $A_4$ are not equivalent. 
\end{proof}

\subsection{$4 \times 4 \times 4$ symmetric tensors}\label{sec:examples_444}

We present two examples in four dimensions below.

\begin{example}\label{ex:444_1}
\emph{We construct the $4 \times 4 \times 4$ real symmetric tensor $A_5$ as $A_5=A_1\oplus I_{(2)}$, where
\begin{equation*}
A_5=\left( \begin{pmatrix}
1 & 0 & 0 & 0\\
0 & -1 & 0 & 0\\
0 & 0 & 0 & 0 \\
0 & 0 & 0 & 0 \\
\end{pmatrix} , \begin{pmatrix}
0 & -1 & 0 & 0 \\
-1 & 0 & 0 & 0 \\
0 & 0 & 0 & 0 \\
0 & 0 & 0 & 0 \\
\end{pmatrix}, \begin{pmatrix}
0 & 0 & 0 & 0 \\
0 & 0 & 0 & 0 \\
0 & 0 & 1 & 0 \\
0 & 0 & 0 & 0 \\
\end{pmatrix}, \begin{pmatrix}
0 & 0 & 0 & 0 \\
0 & 0 & 0 & 0 \\
0 & 0 & 0 & 0 \\
0 & 0 & 0 & 1 \\
\end{pmatrix} \right).
\end{equation*}
According to Definition~\ref{def:eigenvalue}, its eigenvalues $\lambda \in \R$ are computed by solving 
 \begin{equation}\label{eqn:ex5_system}
  \begin{cases} 
    x^2-y^2=\lambda x, \\
    -2xy=\lambda y, \\
    z^2 = \lambda z, \\
    s^2 = \lambda s, \\
    x^2+y^2+z^2+s^2=1.
  \end{cases}
 \end{equation}}

\emph{If $z = 0, s = 0$, System~\eqref{eqn:ex5_system} is equivalent to System~\eqref{eqn:ex1_system} and we get three eigenvalues $\lambda_{1}^2=\lambda_{2}^2=\lambda_{3}^2=1$. If $z \ne 0, s = 0$ or $z = 0, s \ne 0$, it is obvious that System~\eqref{eqn:ex5_system} is equivalent to System~\eqref{eqn:ex3_system} with $z=\lambda$. Thus, each case will give us four eigenvalues, that is, $\lambda_4^2=\lambda_5^2=1, \ \lambda_6^2=\cdots=\lambda_{11}^2=\frac{1}{2}.$ If $z \ne 0, s \ne 0$, we have $z = s = \lambda$ and System~\eqref{eqn:ex5_system} is reduced to a system in two variables. Moreover, the normalization equation becomes $x^2+y^2-2\lambda^2=1$. We use the combination resultant to compute another four eigenvalues as $\lambda_{12}^2=\frac{1}{2},\lambda_{13}^2=\lambda_{14}^2=\lambda_{15}^2=\frac{1}{3}$.}
\end{example}

\begin{example}\label{ex:444_2}
\emph{We construct the $4 \times 4 \times 4$ real symmetric tensor $A_6$ as $A_6=A_2\oplus I_{(2)}$, where
\begin{equation*}
A_6=\left( \begin{pmatrix}
1 & 0 & 0 & 0\\
0 & \frac{1}{2} & 0 & 0\\
0 & 0 & 0 & 0 \\
0 & 0 & 0 & 0 \\
\end{pmatrix} , \begin{pmatrix}
0 & \frac{1}{2} & 0 & 0 \\
\frac{1}{2} & 0 & 0 & 0 \\
0 & 0 & 0 & 0 \\
0 & 0 & 0 & 0 \\
\end{pmatrix}, \begin{pmatrix}
0 & 0 & 0 & 0 \\
0 & 0 & 0 & 0 \\
0 & 0 & 1 & 0 \\
0 & 0 & 0 & 0 \\
\end{pmatrix}, \begin{pmatrix}
0 & 0 & 0 & 0 \\
0 & 0 & 0 & 0 \\
0 & 0 & 0 & 0 \\
0 & 0 & 0 & 1 \\
\end{pmatrix} \right).
\end{equation*}
According to Definition~\ref{def:eigenvalue}, we obtain eigenvalues $\lambda \in \R$ of $A_6$ by solving 
 \begin{equation}\label{eqn:ex6_system}
  \begin{cases} 
    x^2+\frac{1}{2}y^2=\lambda x, \\
    xy=\lambda y, \\
    z^2 = \lambda z, \\
    s^2 = \lambda s, \\
    x^2+y^2+z^2+s^2=1.
  \end{cases}
 \end{equation}}

\emph{If $z = 0, s = 0$, System~\eqref{eqn:ex5_system} is equivalent to System~\eqref{eqn:ex1_system}, which produces three eigenvalues $\lambda_{1}^2=\lambda_{2}^2=\lambda_{3}^2=1$. If $z \ne 0, s = 0$ or $z = 0, s \ne 0$, it is obvious that System~\eqref{eqn:ex5_system} is equivalent to System~\eqref{eqn:ex3_system} with $z=\lambda$. Then we obtain eight eigenvalues in total as $\lambda_4^2=\lambda_5^2=1, \ \lambda_6^2=\cdots=\lambda_{11}^2=\frac{1}{2}$. If $z \ne 0, s \ne 0$, it follows that $z = s = \lambda$ and System~\eqref{eqn:ex5_system} only has two variables. Using the combination resultant, we obtain the last four eigenvalues as $\lambda_{12}^2=\frac{1}{2},\lambda_{13}^2=\lambda_{14}^2=\lambda_{15}^2=\frac{1}{3}$.}
\end{example}

We proceed to extend our results to four dimensions.

\begin{proposition}\label{prop:tensor444}
The tensors $A_5$ and $A_6$ are not equivalent.
\end{proposition}


\begin{proof}
The proof is similar to the proof of Proposition~\ref{prop:tensor333}.
\end{proof}

We have provided explicit examples of two-, three-, and four-dimensional real symmetric tensors that have the same eigenvalues but belong to different equivalence classes. This result is readily generalized to higher dimensions as follows. 

\begin{theorem}\label{prop:tensornnn}
Let $T_{i} = A_i\oplus I_{(n-2)}$ for $i=1,2$. Then $T_1$ and $T_2$ have the same eigenvalues, but they are not equivalent.
\end{theorem}

\begin{proof}
Recall that the $A_1$ and~$A_2$ in Examples \ref{ex:222_1} and \ref{ex:222_2} are $2 \times 2 \times 2$ real symmetric tensors that share identical eigenvalues but are not equivalent. Similar to the $3 \times 3 \times 3$ and $4 \times 4 \times 4$ cases, the eigenvalues of $T_1$ and $T_2$ can be obtained by decomposing the corresponding System~\eqref{def:eigenvalue} and considering the eigenvalues of $A_i \oplus I_{(k)}$ for all $0 \le k \le n-2$. Hence, if $A_1$ and $A_2$ have identical eigenvalues, then so do $T_1$ and $T_2$. 

The tensors $T_1$ and $T_2$ are associated with the polynomials $f_1(x_1,\ldots,x_n) = x_1^3-3x_1x_2^2+x_3^3+\ldots+x_n^3$ and $f_2(x_1,\ldots,x_n) = x_1^3+\frac{3}{2}x_1x_2^2+x_3^3+\ldots+x_n^3$, respectively. It follows from Theorem~\ref{thm:center} that $Z(f_1) \cong \C \times \R \times \cdots \times \R$  ($n-2$ times) and $Z(f_2) \cong \R \times \cdots \times \R$ ($n$ times). Since the centers of $f_1$ and $f_2$  are not isomorphic, their associated tensors $T_1$ and $T_2$ are not equivalent.
\end{proof}

\begin{remark}
In this section, we take advantage of the direct sum to construct examples. For a direct sum tensor, it is clear by the definition that its eigenvalues contain the union of the direct summands'. However, there are many more than that, see e.g. the case of the unit tensor \cite{cartwright2013number}. It seems of independent interest to fully uncover all the eigenvalues or the eigenpairs of a direct sum tensor from those of its direct summands. 
\end{remark}

\section*{Use of AI tools declaration}
The authors declare they have not used Artificial Intelligence (AI) tools in the creation of this article.

\section*{Acknowledgements}
L. Fang was partially supported by the Natural Science Foundation of Xiamen (Grant No. 3502Z202371014).

H.-L. Huang was partially supported by the Key Program of the Natural Science Foundation of Fujian Province (Grant No. 2024J02018) and the National Natural Science Foundation of China (Grant No. 12371037).

Y. Ye was partially supported by the National Key R\&D Program of China (Grant No. 2024YFA1013802), the National Natural Science Foundation of China (Grant Nos. 12131015 and 12371042), and the Innovation Program for Quantum Science and Technology (Grant No. 2021ZD0302902).

\section*{Conflict of interest}
The authors declare no conflicts of interest.

\begin{spacing}{.88}
\setlength{\bibsep}{2.pt}
\bibliographystyle{abbrvnat}
\bibliography{tensor_eigenvalue}

@article{huang2021diagonalizable,
	author = {Hua-Lin Huang and Huajun Lu and Yu Ye and Chi Zhang},
	title = {Diagonalizable higher degree forms and symmetric tensors},
	journal = {Linear Algebra and its Applications},
	volume = {613},
	pages = {151-169},
	year = {2021},
	doi = {10.1016/j.laa.2020.12.018}
}

@article{huang2022centres,
  author = {Hua-Lin Huang and Huajun Lu and Yu Ye and Chi Zhang},
  title = {On centres and direct sum decompositions of higher degree forms},
  journal = {Linear and Multilinear Algebra},
  volume = {70},
  number = {22},
  pages = {7290-7306},
  year = {2022},
  publisher = {Taylor & Francis},
  doi = {10.1080/03081087.2021.1985057}
}

@article{qi2005eigenvalues,
  title={Eigenvalues of a real supersymmetric tensor},
  author={Qi, Liqun},
  journal={Journal of Symbolic Computation},
  volume={40},
  number={6},
  pages={1302--1324},
  year={2005},
  publisher={Elsevier},
  doi = {10.1016/j.jsc.2005.05.007}
}

@article{qi2006rank,
  title={Rank and eigenvalues of a supersymmetric tensor, the multivariate homogeneous polynomial and the algebraic hypersurface it defines},
  author={Qi, Liqun},
  journal={Journal of Symbolic Computation},
  volume={41},
  number={12},
  pages={1309--1327},
  year={2006},
  publisher={Elsevier},
  doi = {10.1016/j.jsc.2006.02.011}
}

@article{qi2007eigenvalues,
  title={Eigenvalues and invariants of tensors},
  author={Qi, Liqun},
  journal={Journal of Mathematical Analysis and Applications},
  volume={325},
  number={2},
  pages={1363--1377},
  year={2007},
  publisher={Elsevier},
  doi={10.1016/j.jmaa.2006.02.071}
}

@article{cartwright2013number,
  title={The number of eigenvalues of a tensor},
  author={Cartwright, Dustin and Sturmfels, Bernd},
  journal={Linear Algebra and its Applications},
  volume={438},
  number={2},
  pages={942--952},
  year={2013},
  publisher={Elsevier},
  doi={10.1016/j.laa.2011.05.040}
}

@article{li2013characteristic,
  title={E-characteristic polynomials of tensors},
  author={Li, An-Min and Qi, Liqun and Zhang, Bin},
  journal={Communications in Mathematical Sciences},
  volume={11},
  number={1},
  pages={33--53},
  year={2013},
  publisher={International Press of Boston, Inc.},
  doi={10.4310/CMS.2013.v11.n1.a2}
}

@article{hu2013characteristic,
  title={The {E}-characteristic polynomial of a tensor of dimension 2},
  author={Hu, Shenglong and Qi, Liqun},
  journal={Applied Mathematics Letters},
  volume={26},
  number={2},
  pages={225--231},
  year={2013},
  publisher={Elsevier},
  doi={10.1016/j.aml.2012.08.017}
}

@article{harrison1975grothendieck,
	title = {A grothendieck ring of higher degree forms},
	journal = {Journal of Algebra},
	volume = {35},
	number = {1},
	pages = {123-138},
	year = {1975},
	doi = {10.1016/0021-8693(75)90039-3},
	author = {David K Harrison}
}

@article{ni2007degree,
  title={The degree of the {E-characteristic} polynomial of an even order tensor},
  author={Ni, Guyan and Qi, Liqun and Wang, Fei and Wang, Yiju},
  journal={Journal of Mathematical Analysis and Applications},
  volume={329},
  number={2},
  pages={1218--1229},
  year={2007},
  publisher={Elsevier}
}

@book{cox2005using,
  title={Using Algebraic Geometry},
  author={Cox, David A and Little, John and O'shea, Donal},
  volume={185},
  year={2005},
  publisher={Springer Science \& Business Media}
}

@book{qi2018tensor,
  title={Tensor Eigenvalues and Their Applications},
  author={Qi, Liqun and Chen, Haibin and Chen, Yannan},
  volume={39},
  year={2018},
  publisher={Springer}
}

@book{yang1996nonlinear,
  title={Nonlinear algebraic equation system and automated theorem proving},
  author={Yang, Lu and Jing-Zhong Zhang and Xiao-Rong Hou},
  year={1996},
  publisher={Shanghai Scientific and Technological Education Publishing House}
}

@article{zhou1999preliminary,
  title={The preliminary of combination resultant approach on polynomial equation system},
  author={Zhou, Jia-Nong},
  journal={Journal of Sichuan University (Natural Science Edition)},
  volume={36},
  number={2},
  pages={206--210},
  year={1999}
}

@article{huang2025solving,
  title={Solving algebraic equations by completing powers},
  author={Huang, Hua-Lin and Ruan, Shengyuan and Xu, Xiaodan and Ye, Yu},
  journal={The American Mathematical Monthly},
  volume={132},
  number={3},
  pages={218--236},
  year={2025},
  publisher={Taylor \& Francis},
  doi={10.1080/00029890.2024.2421145}
}

@article{huang2022harrison,
  title={Harrison center and products of sums of powers},
  author={Huang, Hua-Lin and Liao, Lili and Lu, Huajun and Ye, Yu and Zhang, Chi},
  journal={Communications in Mathematics and Statistics},
  year={2023},
  publisher={Springer},
  doi={10.1007/s40304-023-00367-1}
}

@book{landsberg2017geometry,
  title={Geometry and Complexity Theory},
  author={Landsberg, Joseph M},
  volume={169},
  year={2017},
  publisher={Cambridge University Press},
  doi={10.1017/9781108183192}
}

@article{fang2025simultaneous,
  title={Simultaneous direct sum decompositions of several multivariate polynomials},
  author={Fang, Lishan and Huang, Hua-Lin and Liao, Lili},
  journal={Linear Algebra and its Applications},
  volume={724},
  pages={320--335},
  year={2025},
  publisher={Elsevier},
  doi={10.1016/j.laa.2025.06.022}
}

@inproceedings{acuaviva2023minimal,
  title={The minimal canonical form of a tensor network},
  author={Acuaviva, Arturo and Makam, Visu and Nieuwboer, Harold and P{\'e}rez-Garc{\'\i}a, David and Sittner, Friedrich and Walter, Michael and Witteveen, Freek},
  booktitle={2023 IEEE 64th Annual Symposium on Foundations of Computer Science (FOCS)},
  pages={328--362},
  year={2023},
  organization={IEEE},
  doi={10.1109/FOCS57990.2023.00027}
}

@article{dur2000three,
  title={Three qubits can be entangled in two inequivalent ways},
  author={D{\"u}r, Wolfgang and Vidal, Guifre and Cirac, J Ignacio},
  journal={Physical Review A},
  volume={62},
  number={6},
  pages={062314},
  year={2000},
  publisher={APS},
  doi={10.1103/PhysRevA.62.062314}
}

@book{horn2012matrix,
  title={Matrix Analysis},
  author={Horn, Roger A and Johnson, Charles R},
  year={2012},
  publisher={Cambridge University Press}
}

@inproceedings{lim2005singular,
  title={Singular values and eigenvalues of tensors: a variational approach},
  author={Lim, Lek-Heng},
  booktitle={1st IEEE International Workshop on Computational Advances in Multi-Sensor Adaptive Processing},
  pages={129--132},
  year={2005},
  organization={IEEE},
  doi={10.1109/CAMAP.2005.1574201}
}

@inproceedings{yuan2010computing,
  title={Computing the {D}ixon derived polynomial by combination resultant method},
  author={Yuan, Xun and Zhang, Jingzhong and Feng, Yong},
  booktitle={2010 International Conference on Computational and Information Sciences},
  pages={681--684},
  year={2010},
  organization={IEEE},
  doi={10.1109/ICCIS.2010.169}
}

\end{spacing}
\end{document}